\documentclass[10pt]{amsart}

\usepackage{amssymb,amsthm,amsmath}
\usepackage[numbers,sort&compress]{natbib}
\usepackage{color}
\usepackage{graphicx}
\usepackage{tikz}

\title[ ]{Upper  bounds  on the spectral gaps of quasi-periodic Schr\"odinger  operators  with Liouville frequencies}

\author{Wencai Liu}
\address[Wencai Liu]{Department of Mathematics, University of California, Irvine, California 92697-3875, USA} \email{liuwencai1226@gmail.com}

\author{Yunfeng Shi}
\address[Yunfeng Shi]{School of Mathematical Sciences,
Fudan University,
Shanghai 200433,
P. R. China}
\address[]{School of Mathematical Sciences,
Peking University,
Beijing 100871,
P. R. China}
 \email{yunfengshi13@fudan.edu.cn}

\theoremstyle{plain}
\newtheorem{thm}{Theorem}[section]
 
 \newtheorem{lem}[thm]{Lemma}
 
 \theoremstyle{definition}
 
 \theoremstyle{remark}
 \newtheorem{rem}[thm]{Remark}
 
 \numberwithin{equation}{section}
\begin{document}


\begin{abstract}
 We prove that the size of  the spectral  gaps of  weakly coupled quasi-periodic Schr\"odinger  operators with Liouville frequencies decays exponentially. As an application, we   obtain the    homogeneity of the spectrum.
\end{abstract}

\maketitle
\section{Introduction and main results}
Let us consider a  quasi-periodic Schr\"odinger  operator given   by
\begin{equation}\label{qpso}
(H_{\lambda f,\alpha,\theta}x)_n=x_{n+1}+x_{n-1}+\lambda f(\theta+n\alpha)x_n,
\end{equation}
 where ${x=\{x_n\}_{n\in\mathbb{Z}}}\in l^2(\mathbb{Z})$  and  $f$ is a real analytic function on $\mathbb{R}/\mathbb{Z}$. We will refer to such operators as QPS. The QPS depend on three parameters $(\alpha,
\lambda,\theta)\in \mathbb{R}^3$. Usually,  we call
  $\alpha $  the  frequency, $\theta $  the  phase and $\lambda$   the coupling.  In particular, if $f(x)=2\cos{(2\pi x)}$,  we call (\ref{qpso}) an almost Mathieu operator (AMO). We will
   denote  AMO by  $H_{\lambda,\alpha,\theta}$ .

 For rational frequency $\alpha$, the spectrum consists  of finite number of intervals by Floquet theory.
For irrational frequency $\alpha$ and nonzero coupling constant,
it is well-known that the spectrum does not depend on $\theta$  and we denote it by $ \Sigma_{\lambda f,\alpha}$ ( $ \Sigma_{\lambda,\alpha}$ for AMO).
From now on, we always assume $\alpha$ is irrational.
Each    connected component of  $[E_{\text{min}},E_{\text{max}}]\backslash \Sigma_{\lambda f,\alpha}$ is called    a   spectral  gap, where $ E_{\text{min}} =\min\{E: E\in \Sigma_{\lambda f,\alpha}\} $ and $ E_{\text{max}} =\max\{E: E\in \Sigma_{\lambda f,\alpha}\}$.
If $\lambda=0$,  the spectrum $\Sigma_{\lambda f,\alpha}=[-2,2]$  so that  there is no spectral gap.
  Then it is interesting to  study  upper bounds of the size of the    spectral gaps under small perturbation (i.e. $\lambda$ is small).
In order to state the results,
we introduce the   fibered rotation number  $\rho_{\lambda f,\alpha}(E)$ (see section 2.3) of  QPS, which has the following properties \cite{JM1982}:
\begin{description}
  \item [(i)] $\rho_{\lambda f,\alpha}(\cdot)$ is a continuous non-increasing surjective function with
$\rho_{\lambda f,\alpha}\text{:\;}\mathbb{R} \rightarrow[0,\frac{1}{2}]$, $\rho_{\lambda f,\alpha}(E)=\frac{1}{2}$ for $E\leq E_{\text{min}}$ and $\rho_{\lambda f,\alpha}(E)=0$ for $E\geq E_{\text{max}}$.
  \item [(ii)] For each spectral gap $G$, there exists a unique integer   $m\in \mathbb{Z}\backslash \{0\}$, such that the fibered rotation
number restricted to the spectral gap satisfies $ 2\rho_{\lambda f,\alpha}|_G\equiv m\alpha\text{ mod } \mathbb{Z}$.
\end{description}

For any $m\in \mathbb{Z}\backslash \{0\}$, let us define
\begin{equation}\label{liugap}
    [E_m^-,E_m^+]=\{E\in \mathbb{R}:2\rho_{\lambda f,\alpha}(E)=  m\alpha\text{ mod } \mathbb{Z} \}.
\end{equation}
Now we distinguish  two cases.
\begin{itemize}
 \item[(1)] $E_m^-<E_m^+$. In this case, letting $G_m=(E_{m}^-,E_{m}^+)$, then $G_m$ is a   spectral gap.
  \item[(2)]  $E_m^-=E_m^+$. In this case,  let $G_m=\{E_{m}^-\}$ and we call  $G_m=\{E_{m}^-\}$   a collapsed spectral gap.
\end{itemize}
Thus  for any $m\in \mathbb{Z}\backslash \{0\}$, there exists a corresponding (possibly collapsed) spectral gap  $G_m$.

We give some history of the results on  the lower bounds of the spectral gaps first, which originate  from the  study of  the dry  Ten Martini Problem.
The spectrum $\Sigma_{\lambda,\alpha}$ of AMO has been conjectured to be a
Cantor set (dubbed  the   Ten Martini Problem), which was finally proved by Avila-Jitomirskaya  \cite{Ten}, after an a.e. result by Puig \cite{Puig}.
The dry Ten Martini Problem asserts   that AMO   has no collapsed spectral gap for all $\lambda\neq 0$ and $\alpha\in \mathbb{R}\backslash \mathbb{Q}$, which
is stronger than the Ten Martini Problem by   properties of  the fibered  rotation number $\rho_{\lambda,\alpha}$.
In \cite{CEY},  Choi-Elliott-Yui showed that AMO has no collapsed spectral gap by setting up the lower bounds of the spectral gaps if  $\alpha$ is Liouville and $\lambda$ satisfies some assumption.
Here,
 $\alpha$ is Liouville means  $\beta(\alpha)>0$, where
 \begin{equation}\label{sde}
\beta(\alpha)=\limsup_{k\rightarrow \infty}\frac{-\ln{||k\alpha||_{\mathbb{R}/\mathbb{Z}}}}{|k|},
\end{equation}
and $||x||_{\mathbb{R}/\mathbb{Z}}=\min\limits_{k\in\mathbb{Z}}|x-k|$.
To  the contrary, if $\beta(\alpha)=0$,
 we say $\alpha$ is weak Diophantine  \footnote{We say $\alpha$ is   Diophantine if there exists $\kappa,\tau>0$ such that $||k\alpha||_{\mathbb{R}/\mathbb{Z}}\geq \frac{\tau}{|k|^{\kappa}} $ for any $k\neq0$.}.

 Later, Puig proved the dry Ten Martini Problem for a set of $(\alpha,\lambda)$ of  positive Lebesgue measure \cite{Puig,Puig06}.
 Puig approached it by conjugating the Schr\"odinger cocycle $S_{\lambda f,E}$ to a parabolic matrix $\left[
                                                                                          \begin{array}{cc}
                                                                                            \pm 1 & \mu \\
                                                                                            0 & \pm 1 \\
                                                                                          \end{array}
                                                                                        \right]
 $ and perturbing  $S_{\lambda f,E}$ to   $S_{\lambda f,E + \varepsilon}$, where $E=E_m^-$ or $E=E_m^+$.
 This idea is  significantly developed    by Avila and Jitomirskaya  \cite{AJa1,AJCMP,AA2008}, in which  they were able to deal with all Diophantine $\alpha$ and $\lambda\neq \pm 1$. Avila-Jitomirskaya's reducibility result  also holds for  general analytical potentials  with small coupling constant.
 For the quantitative lower bounds of the spectral gaps,  see \cite{Leguil1,krasovsky2016central} and the references therein.

 Now, let us move to the upper bounds on   spectral gaps.
 Moser and P\"oschel have shown that for a small analytic potential
and a Diophantine vector of frequencies,    the spectral gap with  some certain label $k$  decays exponentially.
 For  the  continuous quasi-periodic operators, the breakthrough is from  Damanik and Goldstein where they
  obtained very precise exponential decay in terms of the smallness
of  potentials, and   the decay rate is  bounded  by the size of the
analytic strip \cite{dam}.
As an application, the homogeneity of the spectrum can also be obtained \cite{dam1}.
 For the discrete case,
Leguil-You-Zhao-Zhou   \cite{Leguil1} showed that the rotation number $\rho_{\lambda f,\alpha}(E+\varepsilon)$  of $S_{\lambda f,E + \varepsilon}$ will change under large perturbation $\varepsilon$ based on the reducibility result of   \cite{AJa1}, where $E=E_m^-$ or $E=E_m^+$. This leads to an upper bound of the spectral gap.
We will say more after the statements of our main results.
  Before that,
  Amor \cite{Amor} got an upper bound on the spectral gap  for  small coupling constant and  Diophantine frequency by KAM theory stemming  from \cite{eli}.  
Finally, we mention that
 there are some direct results about the   homogeneity of the spectrum, see \cite{gold2015,gold2016,fill}.

 However, all of  the results for  general analytic potentials on the  upper bounds of   spectral gaps   are focused  on  (weak) Diophantine frequencies.
 Recently, there has been a significant interest in extending various Diophantine results to the case of Liouville frequencies, as phase transitions in the behaviors of various objects happen in this regime \cite{ayzduke,ruij,avila2016second,jl18,jl18r,Ji}. The contribution of the present paper is to investigate the upper bounds on gaps for the  Liouville frequencies.
\begin{thm}\label{main}
Let  $H_{\lambda f,\alpha,\theta}$ be given by \eqref{qpso} and   $E_m^-, E_m^+$  be given by \eqref{liugap}.  Suppose  $\alpha\in \mathbb{R}\setminus \mathbb{Q}$ satisfies $0\leq \beta(\alpha)<\infty$.
Then there exists an absolute constant  $C>0 $  such that if $f$ is analytic on the strip $\{x\in\mathbb{C}/\mathbb{Z}:|\Im{x}|<h\}$ and $\beta(\alpha)\leq \frac{h}{C^2}$,  then there exist  $\lambda_0=\lambda_0(f,h,\beta(\alpha))>0$ and  $m_{\star}= m_{\star}(\lambda,f,h,\alpha)>0$ such that for any $|\lambda|\leq \lambda_0$, the following estimate holds
\begin{equation*}\label{epd}
E_m^+-E_m^-\leq e^{-\frac{h}{C}|m|}
\end{equation*}
  for $|m|\geq m_{\star}$. In the  particular case of AMO, $\lambda_0=e^{-C^2h}$.
\end{thm}
For trigonometric polynomial  potential, one has
\begin{thm}\label{maint}
Let  $H_{\lambda f,\alpha,\theta}$ be given by \eqref{qpso} and   $E_m^-, E_m^+$  be given by \eqref{liugap}.  Suppose  $\alpha\in \mathbb{R}\setminus \mathbb{Q}$ satisfies $0\leq \beta(\alpha)<\infty$ and f is a trigonometric polynomial.
Then for any $\eta>0$, there exist $\lambda_0=\lambda_0(f,\eta,\beta(\alpha))>0$ and  $m_{\star}= m_{\star}(\lambda,f,\eta,\alpha)>0$, such that for  $|\lambda|\leq \lambda_0$, the following estimate
holds
\begin{equation*}\label{epd}
E_m^+-E_m^-\leq e^{-\eta|m|}
\end{equation*}
 for $|m|\geq m_{\star}>0$.
\end{thm}

For AMO, we have the following refinement.
\begin{thm}\label{main1}
Let  $H_{\lambda,\alpha,\theta}$ be an almost Mathieu operator and   $E_m^-, E_m^+$  be given by \eqref{liugap}.  Suppose  $\alpha\in \mathbb{R}\setminus \mathbb{Q}$ satisfies $0\leq \beta(\alpha)<\infty$.
Then there exists an absolute constant  $C>0 $ such that  for  any $|\lambda|\leq e^{-C\beta(\alpha)}$, the following estimate
holds
\begin{equation}\label{epd}
E_m^+-E_m^-\leq |\lambda|^{\frac{1}{C}|m|}
\end{equation}
 for $|m|\geq  m_{\star}(\lambda, \alpha)$.
\end{thm}

\begin{rem}
(1)   The case of   $\beta(\alpha)=0$ in Theorem \ref{main1} with explicit $C$  was proved in   \cite{Leguil1}.\par
(2) Under the assumption $0\leq |\lambda|\leq e^{-C\beta(\alpha)}$, Liu and Yuan \cite{LYJFG} proved that there is no collapsed spectral gap, i.e., $E_m^+>E_m^-$
for any nonzero integer $m$.
\par
(3). Under the condition of Theorem \ref{main1}, we have $ |\lambda|^{\frac{1}{C}|m|}\leq e^{\frac{\ln |\lambda|}{C}|m|}$. So the
size of the spectral gap $G_m=(E_{m}^-,E_{m}^+)$ decays exponentially with respect to the label $m$.
\par
(4). We expect the optimal decay in \eqref{epd} to be $C|\lambda|^{|m|}$.

\end{rem}
As an application, we obtain
\begin{thm}\label{main2}
 Under the condition of Theorem \ref{main}, for any $\epsilon>0$, there exists $\sigma_{\star}=\sigma_{\star}(\lambda,f,\alpha,\epsilon)>0$  such that for all $E\in\Sigma_{\lambda f,\alpha}$ and $ \sigma\in(0,\sigma_{\star})$, we have
\begin{equation*}\label{hom}
\mathrm{Leb}\left((E-\sigma,E+\sigma)\cap\Sigma_{\lambda f,\alpha}\right)\geq(1-\epsilon)\sigma,
\end{equation*}
where $\mathrm{Leb}(\cdot)$ is the Lebesgue measure.
\end{thm}
\begin{rem}
By letting  $E$ be a point on the  boundary of a spectral gap, we see that the lower bound  $1-\epsilon$ is optimal.
\end{rem}
 We want to explain the motivations for results, and also explain the new
challenge for the Liouville case.
Recently,   the global theory of one-frequency cocycles has been
proposed.  The spectrum of the quasi-periodic operator (or the corresponding Schr\"odinger cocycle) can be classified into three regimes:
\begin{itemize}
  \item Supercritical regime if the Lyapunov exponent is positive.
  \item  Subcritical if the corresponding  transfer matrices $A_{n}(z)$ are uniformly subexponentially
bounded through some strip  $|\Im z|\leq h$.
  \item  Critical regime  otherwise.
\end{itemize}
See \cite{Global,avila2010almost} the formal definition and generalization.
The three regimes have very important spectral features.
Roughly speaking, the (almost) reducibility in subcritical regime is the competition between $h$ and $\beta(\alpha)$ and it relates to absolutely continuous spectrum
\cite{eli,Dina,Houyou,avila2010almost,AA2008,AFK,Amor}.
The (almost) localization in supercritical regime is the competition between the positive Lyapunov exponent and resonance (it is governed by the frequency resonance $\beta(\alpha)$ and the phase resonance) and it relates to the singular continuous spectrum and the pure point spectrum \cite{Ten,Ji99,jl18,jl18r}.  The critical regime relates to the singular continuous spectrum \cite{AJM,AK06}.
 The supercritical regime and  subcritical regime can be connected by Aubry duality, and then the (almost) localization and (almost) reducibility are connected \cite{Puig,Puig06,ayzduke,Ji,LYJFG}. However, most of the previous references focus on Diophantine frequencies.
 The motivation of the results in this paper is to set up the {\it quantitative} almost reducibility by the almost localization in the dual model so that we can deal with upper bounds of  spectral gaps.  Roughly  speaking, in order to  balance  the small divisor from the frequency $\alpha$, we need the subcritical regime at least in  a strip of width $h>C^2\beta(\alpha)$ and the  upper bounds of   spectral gaps are controlled by  the decaying rate $\gamma=\frac{h}{C}$, where $C$ is a large absolute constant.
 In this paper, we do not focus on the explicit value of $C$ but it is doable. In particular,  we only need $h>0$ in the case of  weak Diophantine frequencies ($\beta(\alpha)$=0).
 It is very difficult to address the spectral gap by the approach $\beta(\alpha)\to h\to \gamma$.
  As aforementioned,  the  recent results for general analytic  potentials  are  to deal with  Diophantine frequencies  \cite{Leguil1}.
   For one dimensional case with general analytic potentials and weak Diophantine frequencies ($\beta(\alpha)=0$), the authors obtained that  $\gamma>0$ depends on the strip width $h$  in \cite{Leguil1}. We are able to give the explicit formula for $\gamma$ for any frequency with $\beta(\alpha)<\infty$, that is $\gamma=\frac{h}{C}$.
    We should mention that the results in  \cite{Leguil1} hold in higher dimensions and the explicit $\lambda_0$ is given. The most challenged part in this paper is to deal with Liouville  frequencies ($\beta(\alpha)>0$).
 The problems of   Liouville  frequencies are very  hard to deal with. The traditional KAM theory is not able to set up  the reducibility for the corresponding cocycle.
 Recently, there are several  big progresses  to deal with Liouville frequencies \cite{AFK,Houyou,LYJFG,Ten,jl18,ayzduke}. Our method is based on several combinations  of previous methods plus the   delicate quantitative estimate.
  Our  approach from $h\to \gamma $ is inspired by  \cite{Leguil1}. Some challenges related to Liouville frequencies  from $\beta\to h\to \gamma $ have been solved in \cite{LYJFG,LYJMP}, where
 the reducibility results in \cite{AJa1} were extended to    Liouville  frequencies.
Here, we obtain a more delicate and  {\it quantitative} version of the results of  \cite{LYJFG,LYJMP} in order to establish the upper bounds of  the spectral gaps.


The present paper is organized as follows. In section 2, we give some basic concepts and notations.  In section 3, we construct a conjugacy  by Aubry duality in order to reduce the cocycle. In section 4, we perturb the   cocycle near the boundary of  a spectral gap. In section 5, we complete the proofs of Theorems \ref{main}, \ref{maint}, \ref{main1} and \ref{main2}.


\section{Some basic concepts and notations}
\subsection{Cocycle and transfer matrix}
 $C_{\delta}^{\omega}(\mathbb{R}, \mathcal{B})$ be the set of all analytic mappings from $\mathbb{R}$ to some Banach space $(\mathcal{B},||\cdot||)$, which   admit an  analytic extension to the strip $|\Im z|\leq \delta$.  Denote by  $C_{\delta}^{\omega}(\mathbb{R}/\mathbb{Z}, \mathcal{B})\subset C_{\delta}^{\omega}(\mathbb{R}, \mathcal{B})$ the subspace of 1-periodic mappings. Sometimes, we omit $\delta$ for simplicity.
 By a cocycle, we mean a pair $(\alpha,A)\in (\mathbb{R}\setminus\mathbb{Q})\times C_{\delta}^{\omega}(\mathbb{R}/\mathbb{Z},{\rm SL}(2,\mathbb{R}))$
and we can regard it as a dynamical system  on $(\mathbb{R}/\mathbb{Z})\times \mathbb{R}^2$ with
\begin{equation*}
(\alpha,A):(x,v)\longmapsto (x+\alpha, A(x)v),\ (x,v)\in (\mathbb{R}/\mathbb{Z})\times \mathbb{R}^2.
\end{equation*}
For $k>0$, we
define  the $k$-step transfer matrix as
\begin{equation*}
A_k(x)=\prod\limits_{l=k}^{1}A(x+(l-1)\alpha).
\end{equation*}
\subsection{Conjugacy and reducibility}
 Given two cocycles $(\alpha, A)$ and $(\alpha, B)$ with $A,B\in C_{\delta}^{\omega}(\mathbb{R}/\mathbb{Z},{\rm SL}(2,\mathbb{R}))$, a conjugacy between them is a cocycle $(\alpha, R)$ with $R\in C_{\delta}^{\omega}(\mathbb{R}/\mathbb{Z},\text{PSL}(2,\mathbb{R}))$ such that
\begin{equation*}
R^{-1}(x+\alpha)A(x)R(x)=B(x).
\end{equation*}
We say $(\alpha,A)$ is reducible if it   conjugates  to a constant cocycle $(\alpha,B)$.

Given $R\in C^{\omega}(\mathbb{R}/\mathbb{Z},\text{PSL}(2,\mathbb{R}))$, we say the degree of $R$ is $k$ and denote dy $\text{deg}(R)=k$,
if $R$ is homotopic to $  R_{\frac{k}{2}x}$ for some $k\in\mathbb{Z}$, where
\begin{equation*}
R_{x}=\left[
 \begin{array}{cc}
 \cos{2\pi x}&-\sin{2\pi x}\\
\sin{2\pi x}&\cos{2\pi x}\end{array}
\right].
\end{equation*}

\subsection{The fibered rotation number }
  Suppose $A\in C^{\omega}(\mathbb{R}/\mathbb{Z}, \text{SL}(2,\mathbb{R}))$ is homotopic to the identity.
  Then  the fibered rotation number $\rho_{\alpha,A}$ of the cocycle $(\alpha,A)$ is well defined. We refer to papers \cite{AJa1,Leguil1} for the  definition of
  the fibered rotation number.
If $A, B:\mathbb{R} / \mathbb{ Z}\rightarrow \text{SL}(2,\mathbb{R})$   and
$ R:\mathbb{R}/ \mathbb{Z}\rightarrow \text{PSL}(2,\mathbb{R})$   are such
that $A $ is  homotopic to the identity and $ R^{-1}(x+\alpha)A(x)R(x) =B$, then   $ B$ is  homotopic to the identity
and
\begin{equation}\label{pr2}
2\rho(\alpha,A)-2\rho(\alpha,B)=\deg{(R)}\alpha.
\end{equation}
Moreover, there is  some absolute constant $C>0$ such that
\begin{equation}\label{rr}
|\rho(\alpha,A)-\theta|\leq C\sup_{x\in\mathbb{R}/\mathbb{Z}}||A(x)-R_\theta||.
\end{equation}
  In this paper, we consider the Schr\"odinger  cocycle $(\alpha,S_{\lambda f,E})$, where \begin{equation*}
S_{\lambda f,E}(x)=\left[
 \begin{array}{cc}
 E-\lambda f(x)&-1\\
1&0\end{array}
\right].
\end{equation*}
If $f=2\cos(2\pi x)$, we call $(\alpha,S_{\lambda f,E})$ an almost Mathieu cocycle which is denoted by $(\alpha,S_{\lambda,E})$ for simplicity. It is easy to see that  $S_{\lambda f,E}$ is  homotopic to the identity.
Thus the
     fibered rotation number of $(\alpha,S_{\lambda f,E})$ is well defined and denoted by $\rho_{\lambda f,\alpha}(E)$ ($\rho_{\lambda ,\alpha}(E)$ for AMO).
     \subsection{Aubry Duality} For Schr\"odinger operator $H_{\lambda f,\alpha,\theta}$, we define the  dual Schr\"odinger operator   by  $\widehat{H}_{\lambda f,\alpha,\theta}$,
\begin{equation*}
(\widehat{H}_{\lambda f,\alpha,\theta}x)_n=\sum_{k\in\mathbb{Z}}\lambda \widehat{f}_kx_{n-k}+2\cos 2\pi(\theta+n\alpha)x_n,
\end{equation*}
where $\widehat{f}_k$ is the Fourier coefficient of the  potential $f$. Note that the spectrum of  $\widehat{H}_{\lambda f,\alpha,\theta}$ is equal to $\Sigma_{\lambda f,\alpha}$.

 Aubry duality expresses an algebraic relation between the families of
operators $ \{\widehat{H} _{\lambda f,\alpha,\theta}\}_{\theta\in\mathbb{R}} $  and $ \{ {H} _{\lambda f,\alpha,\theta}\}_{\theta\in\mathbb{R}} $
by  Bloch waves, i.e., if
 $ u:\mathbb{R}/\mathbb{Z}\rightarrow\mathbb{C}$ is an $L^2$  function whose Fourier coefficients  $\widehat u$  satisfy
  $ \widehat{H} _{\lambda f,\alpha,\theta}\hat{u}=E\hat{u}$, then
  $\mathcal{U}(x)= \left(
           \begin{array}{c }
             e^{2\pi i \theta }u(x) \\
             u(x-\alpha)\\
           \end{array}
         \right)
  $
satisfies
\begin{equation}\label{dualre}
S_{\lambda f,E}(x)\cdot \mathcal{U}(x)=e^{2\pi i \theta}\mathcal{U}(x+\alpha).
\end{equation}

\subsection{Some notations and assupmtions}
We briefly comment on  the constants and norms in the following proofs. We assume $\alpha$ is irrational and $\lambda\neq 0$.
 We let $C$ (resp. $c$) be large (resp. small) positive absolute constant and $C$ (resp. $c$) may be different even in the same formula.
   $C_2$ (resp. $C_1$) denotes a fixed (resp. any)   constant, which is larger than all the constants $C,c^{-1}$ appearing in this paper.

Let $C(\alpha)$ be  a large constant depending  on $\alpha$ (and $f$)  and $C_{\star}$ (resp. $c_{\star}$)  be a large (resp. small) constant   depending on $\lambda,f$ and $\alpha$. Define for $\delta\geq 0$ the strip $\Delta_{\delta}=\{z\in\mathbb{C}/\mathbb{Z}: |\Im{z}|\leq \delta\}$ and let $||v||_{\delta}=\sup\limits_{\delta\in\Delta_s}||v(z)||$, where $v$ is a mapping from $\Delta_{\delta}$ to some Banach space $(\mathcal{B},||\cdot||)$.
 For any mapping $v$ defined on $\mathbb{R}/\mathbb{Z}$, we let $[v]=\int_{\mathbb{R}/\mathbb{Z}}v(x)\mathrm{d}x$. In this paper, $ \mathcal{B}$ may be
$\mathbb{C}$, $ \mathbb{C}^2$ or $\text{SL}(2,\mathbb{C})$.

\section{The construction of reducibility by Aubry duality}

 In order to state our reducibility result, we introduce some Lemmas first.

\begin{lem}[Theorem 3.3,  \cite{AJa1}]\label{Blol1}
Let $E\in\Sigma_{\lambda f,\alpha}$. Then there exist some $\theta=\theta(E)\in\mathbb{R}/\mathbb{Z}$  and $\widehat{u}=\{\widehat{u}_k\}_{k\in\mathbb{Z}}$ with $\widehat{u}_0=1, |\widehat{u}_k|\leq1 $ such that $\widehat{H}_{\lambda f,\alpha,\theta}\widehat{u}=E\widehat{u}$.
\end{lem}
Suppose $\eta$ satisfies
\begin{equation*}
\eta>C_1\beta(\alpha),
\end{equation*}
where $C_1$ is a large absolute constant.

\begin{lem}[Theorems 4.7 and  5.2, \cite{LYJFG}]\label{Blol2}
Suppose  $\alpha\in \mathbb{R}\setminus \mathbb{Q}$ satisfies $0\leq \beta(\alpha)<\infty$.
Then there exists an absolute constant  $C_2>0 $  such that if $f$ is analytic on the strip $\Delta_{C_2\eta}$,  then there exists  $\lambda_0(f,\eta,\alpha)>0$ (depending only on $f, \eta,\alpha$) such that if $0<|\lambda|\leq \lambda_0(f,\eta,\alpha)$ and $E\in\Sigma_{\lambda f,\alpha}$ with $2\rho_{\lambda f,\alpha}(E)= m\alpha  \mod \mathbb{Z}$,  then there is some $\widetilde{n}\in\mathbb{Z}$ such that
\begin{equation}\label{Bloe1}
 2\theta(E)=\widetilde{n}\alpha\ \mod{\mathbb{Z}},
\end{equation}
and for $|m|\geq m_{\star}$ \footnote{ Recall that $m_{\star}$ is a large constant depending on  $\lambda,f$ and $\alpha$.\label{mstar}}
\begin{equation}\label{p1e3}
|m|\leq C |\widetilde{n}|.
\end{equation}
Moreover,
\begin{equation}\label{ale}
|\widehat{u}_k|\leq C_{\star}e^{-2\pi \eta|k|}, \ \ \mbox{for $|k|\geq 3|\widetilde{n}|$},
\end{equation}
where $\theta(E)$ and $ \{\widehat{u}_k\}$ are given by  Lemma \ref{Blol1}. In particular, $\lambda_0=e^{-C_2\eta}$ for AMO.
\end{lem}
\begin{rem}
The  proof of this lemma for  AMO can be found in \cite{LYJFG}.  It is easy to extend this result to general QPS following the arguments in \cite{LYJMP}. 
\end{rem}

In the following, we fix $\lambda_0$ as in  Lemma \ref{Blol2}.
In order to avoid the repetition, we only give the proof of $\beta(\alpha)>0$. Actually, the proof of $\beta(\alpha)=0$ is much easier.

From now on,
we focus on a  specific  gap $G_m$.
Let   $E=E_m^+\in \Sigma_{\lambda f,\alpha}$  and  $A^{E}(x)=S_{\lambda f,E}(x)$ (sometimes we omit   dependence on  $\lambda$ and $f$ for simplicity).
 We will reduce $(\alpha,A^E)$ to a parabolic matrix $\left[
                                                                                          \begin{array}{cc}
                                                                                            \pm 1 & \mu \\
                                                                                            0 & \pm 1 \\
                                                                                          \end{array}
                                                                                        \right]
 $.
 In order to  study the size of  spectral gap by reducibility, we will set up subtle  estimates on the coefficient $\mu$ and the conjugacy.
  We attach $E$ with $\theta(E)$ and find the localized  solution for the Aubry dual operator.
Then we  use the localized solution given by \eqref{ale} to construct conjugacies which reduce the cocycle.
We always assume the conditions in Lemma \ref{Blol2} are satisfied  so that
\begin{equation*}
n=|\widetilde{n}|<\infty.
\end{equation*}

Our main theorem in this section is
\begin{thm}\label{p1}
Suppose
$0<|\lambda|\leq \lambda_0$. Then for $E=E_m^+$,   there exists $R(x)\in C_{20\beta}^{\omega}(\mathbb{R}/\mathbb{Z},{\rm PSL}(2,\mathbb{R}))$
 such that
\begin{equation}\label{Red}
R^{-1}(x+\alpha)A^{E}(x)R(x) =\left[\begin{array}{cc}\pm1&\mu_m\\
0&\pm1\end{array}\right],
\end{equation}
where
\begin{equation}\label{p1e1}
|\mu_m|\leq C_{\star}e^{-\frac{\eta}{2} n},
\end{equation}
and
\begin{equation}\label{p1e2}
||R||_{20\beta(\alpha)}\leq  C_{\star}e^{C\beta(\alpha)n}.
\end{equation}

\end{thm}
\begin{rem}\label{reliu}
Actually $R$ in Theorem \ref{p1} depends on the label $m$. We ignore the dependence for simplicity.
By some results   in \cite{LYJFG,Leguil1}, we can say more about $\mu_m$,
\begin{description}
  \item [(i)]For general analytic potential $f$, $\mu_m$ may be equal to zero. By Proposition 18 in \cite{Puig06},  the gap $G_m$ is collapsed for $\mu_m=0$ and there is nothing to prove in this case. Thus we assume $\mu_m\neq0$ in the following.
  \item  [(ii)] If $E=E_m^+$ and $\mu_m\neq0$, then the  reduced matrix  can only be
 $ \left[\begin{array}{cc}1&\mu_m\\
0& 1\end{array}\right] $
with $ \mu_m>0$
or
$\left[\begin{array}{cc}-1&\mu_m\\
0& -1\end{array}\right]$
with $ \mu_m<0$ [Theorem 6.1, \cite{LYJFG}].
\end{description}
\end{rem}
 In \cite{LYJFG}, Liu and Yuan got the reducibility \eqref{Red} for AMO  without the estimates of \eqref{p1e1} and \eqref{p1e2}. Thus,
the strategy of the proof of Theorem \ref{p1} is to follow the arguments of Liu and Yuan with quantitative analysis.
For simplicity, we omit the dependence on $m$  in the proof of (\ref{Red}) and (\ref{p1e1}) in this section.

Here, we  give another  lemma, which controls the growth of the cocycle.
\begin{lem}[Theorem 5.1, \cite{LYJMP}]\label{te}
Suppose
$ |\lambda|\leq \lambda_0$. Then
\begin{equation}\label{liu3}
\sup_{0\leq k\leq  e^{\eta n}}||A^{E}_k||_{\eta}\leq C_{\star}e^{C\beta(\alpha)n}.
\end{equation}
\end{lem}

We define
\begin{equation}\label{Bloe3}
\mathcal{U}(x)=\left(\begin{array}{cc}e^{2\pi i\theta}u(x)\\
u(x-\alpha)\end{array}\right),
\end{equation}
where $u(x)=\sum_{k\in\mathbb{Z}}\widehat{u}_ke^{2\pi ikx}$ and $\theta=\theta(E), \widehat{u}=\{\widehat{u}_k\}$ are given by  Lemmas \ref{Blol1} and \ref{Blol2}.

Let
\begin{equation}
\label{new1}\widehat{\mathcal{U}}(x)=e^{i\pi \widetilde{n}x}\mathcal{U}(x).
\end{equation}
\begin{lem}\label{Lemay30}
Let $\widehat{\mathcal{U}}(x)$ be given by (\ref{new1}). Then    $\widehat{\mathcal{U}}(x)$ is  well defined on  $\mathbb{R}/2\mathbb{Z}$ and
analytic  on $ \Delta_{40\beta(\alpha)}$,
and
\begin{equation}\label{Blo2}
||\widehat{\mathcal{U}}||_{40\beta(\alpha)}\leq C_{\star}e^{ C\beta(\alpha) n}.
\end{equation}
\end{lem}
\begin{proof}
This follows from
 \eqref{ale} and  the fact that  $|\widehat{u}_k|\leq1 $ directly.
\end{proof}

\begin{rem}
Actually, $\widehat{\mathcal{U}}(x)$  is analytic
 on $ \Delta_{\eta}$. However $40\beta(\alpha)$ is enough  for our goal.
\end{rem}

By \eqref{dualre}, we have
\begin{equation}\label{se1}
A^{E}(x)\widehat{\mathcal{U}}(x)=\pm\widehat{\mathcal{U}}(x+\alpha).\\
\end{equation}
For the $z\in 40\beta(\alpha)$, define
\begin{equation*}
  \Re{\widehat{\mathcal{U}}(z)}=\frac{\widehat{\mathcal{U}}(z)+\overline{\widehat{\mathcal{U}}(\overline{z})}}{2}; \Im{\widehat{\mathcal{U}}(z)}=\frac{\widehat{\mathcal{U}}(z)-\overline{\widehat{\mathcal{U}}(\overline{z})}}{2i}.
\end{equation*}
Then for  $x\in\mathbb{R}/\mathbb{Z}$, we have
$$\widehat{\mathcal{U}}(x)=\Re{\widehat{\mathcal{U}}(x)}+i\Im{\widehat{\mathcal{U}}(x)}\in \mathbb{R}^2+i\mathbb{R}^2,$$
and it follows from (\ref{se1}) that for $x\in\mathbb{R}/\mathbb{Z}$
\begin{eqnarray}
\label{se2}&&A^E(x)\Re{\widehat{\mathcal{U}}(x)}=\pm\Re{\widehat{\mathcal{U}}(x+\alpha)};\\
\label{se3}&&A^E(x)\Im{\widehat{\mathcal{U}}(x)}=\pm\Im{\widehat{\mathcal{U}}(x+\alpha)}.
\end{eqnarray}
Note that $\Re{\widehat{\mathcal{U}}(x)} $ and $\Im{\widehat{\mathcal{U}}(x)}$ are  well defined on $\mathbb{R}/2\mathbb{Z}$ and  analytic in the strip $\Delta_{40 \beta(\alpha)}$.

\begin{lem}\label{use}
We can select $ \mathcal{V}=\Re{\widehat{\mathcal{U}}}$  or  $\mathcal{V}=\Im{\widehat{\mathcal{U}}}$  such that
 $\mathcal{V}$ is real analytic on $\Delta_{40 \beta(\alpha)}$ and
  \begin{equation}\label{de}
\inf_{|\Im{x}|\leq 40\beta(\alpha)}||\mathcal{V}(x)||\geq c_{\star}e^{-C\beta(\alpha) n}.
 \end{equation}
 \end{lem}
\begin{proof}
Since $\widehat{u}_0=1$, we have
$$||\int_{\mathbb{R}/2\mathbb{Z}}\left(e^{-\widetilde{n}\pi ix}\Re{\widehat{\mathcal{U}}}(x)+ie^{-\widetilde{n}\pi ix}\Im{\widehat{\mathcal{U}}}(x)\right)\mathrm{d}x||=2\sqrt{2}.$$
Thus we can choose  $ \mathcal{V}=\Re{\widehat{\mathcal{U}}}$  or  $\mathcal{V}=\Im{\widehat{\mathcal{U}}}$  such that
\begin{equation}\label{lower}
||\int_{\mathbb{R}/2\mathbb{Z}}e^{-\widetilde{n}\pi ix}\mathcal{V}(x)\mathrm{d}x||\geq\sqrt{2}.
\end{equation}
Suppose  (\ref{de})  is not true. Then there must be some $x_0\in\Delta_{40\beta(\alpha)}$ with $\Im{x_0}=t$ such that
\begin{equation}
||\mathcal{V}(x_0)||\leq  c_{\star}e^{-C\beta(\alpha)n}.
\end{equation}
By Lemma \ref{te}, \eqref{se2} and \eqref{se3}, and
following the arguments of   the  proof of Theorem 4.5 in \cite{LYJFG}, one has
\begin{equation*}
\sup_{x\in\mathbb{R}}||\mathcal{V}(x+it)|| \leq C_{\star}e^{-C\beta(\alpha) n}.
\end{equation*}
Thus, we obtain
\begin{equation*}
||\int_{\mathbb{R}/2\mathbb{Z}}e^{-\widetilde{n}\pi i(x+it)}\mathcal{V}(x+it)\mathrm{d}x||\leq C_{\star}e^{-C\beta(\alpha) n},
\end{equation*}
which   contradicts    (\ref{lower}).

\end{proof}

\textbf{Proof of Theorem \ref{p1}}

\begin{proof}


Let
\begin{equation}\label{se6}
R^{(1)}(x)=\left[\begin{array}{cc}\mathcal{V}(x)&T\frac{\mathcal{V}(x)}{||\mathcal{V}(x)||^2}\end{array}\right],
\end{equation}
where $T\left(
          \begin{array}{c}
            x \\
            y \\
          \end{array}
        \right)
=\left(
          \begin{array}{c}
            -y\\
            x \\
          \end{array}
        \right)$ and $\mathcal{V}$ is given by Lemma \ref{use}.
By Lemma \ref{Lemay30},
it is easy to see that $R^{(1)}\in C_{40\beta}^{\omega}(\mathbb{R}/\mathbb{Z},{\rm PSL}(2,\mathbb{R}))$. From (\ref{Blo2}), (\ref{de}) and (\ref{se6}), we have \begin{equation}\label{se7}
||(R^{(1)})^{-1}||_{40 \beta(\alpha)},||R^{(1)}||_{40 \beta(\alpha)}\leq C_{\star}e^{C\beta(\alpha)n}.
\end{equation}

By (\ref{se2}), (\ref{se3}), (\ref{se6}) and (\ref{se7}), one has
\begin{equation}\label{se8}(R^{(1)})^{-1}(x+\alpha)A^{E}(x)R^{(1)}(x) =\left[\begin{array}{cc}\pm1&\nu(x)\\
0&\pm1\end{array}\right],
\end{equation}
where 
 \begin{equation}\label{se9}
||\nu||_{40\beta(\alpha)}\leq C_{\star}e^{C\beta(\alpha)n}.
\end{equation}

 Now we will reduce the right side hand of (\ref{se8}) to a constant cocycle  by solving a homological equation. More concretely, let $\phi(x)$ be a function defined on $\mathbb{R}/\mathbb{Z}$ with
$[\phi]=0$ and
\begin{equation*}
\left[\begin{array}{cc}1&\phi(x+\alpha)\\
0&1\end{array}\right]^{-1}\left[\begin{array}{cc}\pm1&\nu(x)\\
0&\pm1\end{array}\right]\left[\begin{array}{cc}1&\phi(x)\\
0&1\end{array}\right]=\left[\begin{array}{cc}\pm1&[\nu]\\
0&\pm1\end{array}\right].
\end{equation*}
This can be done if we let
\begin{equation}\label{se11}
\pm\phi(x+\alpha)\mp\phi(x)=\nu(x)-[\nu].
\end{equation}
By considering the Fourier series of  (\ref{se11}), one has
\begin{equation}\label{se12}
\widehat{\phi}_k=\pm \frac{\widehat{\nu}_k}{e^{2\pi ik\alpha}-1}\ (k\neq0),
\end{equation}
where $\widehat{\phi}_k$ and $\widehat{\nu}_k$ are Fourier coefficients of $\phi(x),\nu(x)$ respectively.

 By the definition of $\beta(\alpha)$, we have the following small divisor condition
\begin{equation}\label{small}
 ||k\alpha||_{\mathbb{R}/\mathbb{Z}}\geq C(\alpha)e^{-2\beta(\alpha)|k|}, k\neq0.
\end{equation}
Combining with \eqref{se12} and \eqref{se9}, one has
\begin{equation}\label{se13}
||\phi||_{20\beta(\alpha)}\leq C_{\star}e^{C\beta(\alpha)n}.
\end{equation}
 Let
 \begin{equation}\label{se14}
R(x)=R^{(1)}(x)\left[\begin{array}{cc}1&\phi(x)\\
0&1\end{array}\right].
\end{equation}
By (\ref{se7})  and \eqref{se13},  one has
\begin{equation}\label{liu2}
    ||R||_{20\beta(\alpha)},||R^{-1}||_{20\beta(\alpha)}\leq  C_{\star}e^{C\beta(\alpha)n}.
\end{equation}
This implies (\ref{p1e2}).
Now we are in the position to
give a   estimate on $\mu$. From (\ref{se8}) and (\ref{se14}), we obtain
\begin{equation*}\label{se16}
R^{-1}(x+\alpha)A^{E}(x)R(x) =\left[\begin{array}{cc}\pm1&\mu\\
0&\pm1\end{array}\right],
\end{equation*}
and thus for any $l\in \mathbb{N}$
\begin{equation}\label{se17}
 R^{-1}(x+l\alpha)A_l^{E}(x)R(x)=\left[\begin{array}{cc}\pm1&l\mu\\
0&\pm1\end{array}\right].
\end{equation}
Let $l=l_0=\lfloor  e^{\frac{3}{4}\eta n} \rfloor$ in \eqref{se17}, one has
\begin{eqnarray}
  \nonumber l_0 |\mu| &\leq & ||R^{-1}||_{20\beta(\alpha)}||A^{E}_{l_0}||_{20\beta(\alpha)}||R||_{20\beta(\alpha)} \\
   &\leq&   C_{\star} e^{C\beta(\alpha)n},\label{liu4}
\end{eqnarray}
where the second inequality holds by \eqref{liu3}  and  \eqref{liu2}.

\eqref{p1e1} follows from \eqref{liu4} directly.

\end{proof}
We will give more  details about
\begin{equation}\label{R}
R(x)=\left[\begin{array}{cc}R_{11}(x)&R_{12}(x)\\
R_{21}(x)&R_{22}(x)\end{array}\right],
\end{equation}
which is defined in Theorem \ref{p1}. 

\begin{thm}\label{p2}
Let $ [R_{ij}(x)]_{i,j\in\{1,2\}}$ be  in Theorem \ref{p1}. Then we have
\begin{enumerate}
\item  [(i)]
\begin{eqnarray}
\nonumber&&R_{21}(x+\alpha)=R_{11}(x),\\
\nonumber&&R_{22}(x+\alpha)=R_{12}(x)-\mu R_{11}(x),\\
\label{Re3}&&R_{11}(x+\alpha)R_{12}(x)
-R_{12}(x+\alpha)R_{11}(x)=1+\mu R_{11}(x+\alpha)R_{11}(x);
\end{eqnarray}
 \item[(ii)]
 \begin{equation}\label{Re4}
 [R_{11}^2]=[R_{21}^2]\geq\frac{1}{2||R||_0};
\end{equation}
\item[(iii)]
 \begin{eqnarray}
\label{Re6}&&\frac{[R_{11}^2]}{[R_{11}^2][R_{12}^2]-[R_{11}R_{12}]^2}\leq C_{\star}e^{C\beta(\alpha)n},\\
\label{Re7}&&[R_{11}^2][R_{12}^2]-[R_{11}R_{12}]^2\geq c_{\star}e^{-C\beta(\alpha)n}.
\end{eqnarray}
 \end{enumerate}
\end{thm}

\begin{proof}

(i).  This is done by direct computations and the details can be found in the proof of  Lemma 6.3 in \cite{Leguil1}.

(ii). See  the proof of Lemma 6.2  in \cite{Leguil1}.


(iii). The proof is similar to  that in \cite{Leguil1}. Note
\begin{equation*}
\frac{[R_{11}^2][R_{12}^2]-[R_{11}R_{12}]^2}{[R_{11}^2]}=\left[\left(R_{12}-\frac{[R_{11}R_{12}]}{[R_{11}^2]}R_{11}\right)^2\right],
\end{equation*}
and define
\begin{equation}\label{Re9}
\hat{R}(x)=R_{12}(x)-\frac{[R_{11}R_{12}]}{[R_{11}^2]}R_{11}(x).
\end{equation}
By (\ref{Re3}) and (\ref{Re9}), we have
\begin{equation}\label{Re10}
R_{11}(x+\alpha)\hat{R}(x)-R_{11}(x)\hat{R}(x+\alpha)=1+\mu R_{11}(x+\alpha)R_{11}(x).
\end{equation}
By  Cauchy-Schwarz inequality, one has
\begin{equation}\label{Re11}
\left[|R_{11}(\cdot+\alpha)\hat{R}(\cdot)-R_{11}(\cdot)\hat{R}(\cdot+\alpha)|\right]\leq 2||R||_0\sqrt{[\hat{R}^2]}.
\end{equation}
By (\ref{p1e1}) and (\ref{p1e2}) in Theorem \ref{p1}, we get for
$n\geq n_{\star}$\footnote{$ n_{\star}$ is a large constant depending on $\lambda,f$ and $\alpha$.}
\begin{equation}\label{Re12}
\left[|1+\mu R_{11}(\cdot+\alpha)R_{11}(\cdot)|\right]
\geq\frac{1}{2}.
\end{equation}
By (\ref{Re10}), (\ref{Re11}) and (\ref{Re12}), one has
\begin{equation*}
[\hat{R}^2]\geq\frac{1}{16||R||_0^2}\geq e^{-C\beta(\alpha)n},
\end{equation*}
 which implies (\ref{Re6}).
 Now
 (\ref{Re7})
follows from (\ref{p1e2}), (\ref{Re4}) and (\ref{Re6}).
\end{proof}

\section{Perturbation near the boundary of a spectral gap}

In this section, we will perturb the cocycle $(\alpha,A^{E})$ near the boundary $E=E_m^+$ of a spectral gap $G_m$ with $m\in\mathbb{Z}\setminus\{0\}$.
Without loss of generality, we assume
 the reduced cocycle given by  Theorem \ref{p1} is \begin{equation}\label{pn1}
P=\left[\begin{array}{cc}1&\mu_m\\
0&1\end{array}\right].
\end{equation}
\begin{lem}[ \cite{Puig,Puig06,Leguil1}]\label{pnl1}
Let $R(x)$ be as in Theorem \ref{p1} and $P$ in (\ref{pn1}). Then for $\epsilon \in\mathbb{R}, x\in\mathbb{R}/\mathbb{Z}$, we have
\begin{equation}\label{pn2}
R^{-1}(x+\alpha)A^{E+\epsilon}(x)R(x) =P+\epsilon \widetilde{P}(x),
\end{equation}
where
\begin{equation}\label{pn3}
\widetilde{P}(x)=\left[\begin{array}{cc}R_{11}(x)R_{12}(x)-\mu_m R_{11}^2(x)& R_{12}^2(x)-\mu_m R_{11}(x)R_{12}(x)\\
-R_{11}^2(x)&-R_{11}(x)R_{12}(x)\end{array}\right].
\end{equation}
\end{lem}

Next, we will tackle the perturbed cocycle $(\alpha,P+\epsilon\widetilde{P})$ given in (\ref{pn2}) by   averaging method, which was originally from \cite{MP}, and well developed in  \cite{Puig,Puig06,Amor} for Diophantine frequencies and \cite{LYJFG} for Liouville frequencies.
 We   develop  the   averaging method to reduce the cocycle $(\alpha,P+\epsilon\widetilde{P})$ with the Liouville frequency $\alpha$ to a new constant cocycle plus a   smaller perturbation, that is
\begin{thm}\label{p3}
Let  $\delta=5\beta(\alpha)$. Then  the following statements hold
 \begin{enumerate}
 \item [(i)]for any $|\epsilon|\leq\frac{1}{C(\alpha)||R||_{2\delta}^{2}}$, there exist  $R_{1,\epsilon}, \widetilde{P}_{1,\epsilon}\in C_\delta^{\omega}(\mathbb{R}/\mathbb{Z},{\rm SL}(2,\mathbb{R}))$ and $P_{1,\epsilon}\in {\rm SL}(2,\mathbb{R})$ such that
\begin{equation*}
R_{1,\epsilon}^{-1}(x+\alpha)(P+\epsilon\widetilde{P}(x))R_{1,\epsilon}(x) =P_{1,\epsilon}+\epsilon^2\widetilde{P}_{1,\epsilon}(x)
\end{equation*}
and
\begin{eqnarray}
\label{pn5}&&||R_{1,\epsilon}-I||_\delta\leq C(\alpha)||R||_{2\delta}^2|\epsilon|,\\
\label{pn6}&&||P_{1,\epsilon}-P||\leq C(\alpha)||R||_{2\delta}^2|\epsilon|,\\
\label{pn7}&&||\widetilde{P}_{1,\epsilon}||_\delta\leq C(\alpha)||R||_{2\delta}^4,\\
\label{pn10}&&P_{1,\epsilon}=P+\epsilon[\widetilde{P}];
\end{eqnarray}
\item[(ii)]
for any $|\epsilon|\leq\frac{1}{C(\alpha)||R||_{2\delta}^{4}}$, there exist  $R_{2,\epsilon},\widetilde{P}_{2,\epsilon}\in C^{\omega}(\mathbb{R}/\mathbb{Z},{\rm SL}(2,\mathbb{R}))$ and $P_{2,\epsilon}\in {\rm SL}(2,\mathbb{R})$ such that
\begin{equation}\label{pn11}
R_{2,\epsilon}^{-1}(x+\alpha)(P_{1,\epsilon}+\epsilon^2\widetilde{P}_{1,\epsilon}(x))R_{2,\epsilon}(x) =P_{2,\epsilon}+\epsilon^3\widetilde{P}_{2,\epsilon}(x),\\
\end{equation}
and
\begin{eqnarray}
\label{pn12}&&||R_{2,\epsilon}-I||_0\leq  C(\alpha)||R||_{2\delta}^4\epsilon^2,\\
\label{pn13}\nonumber&&||P_{2,\epsilon}-P_{1,\epsilon}||\leq C(\alpha)||R||_{2\delta}^4\epsilon^2,\\
\label{pn15}\nonumber&&||\widetilde{P}_{2,\epsilon}||_0\leq C(\alpha)||R||_{2\delta}^8,\\
\label{pn16}\nonumber&&P_{2,\epsilon}=P_{1,\epsilon}+\epsilon^2[\widetilde{P}_{1,\epsilon}].
\end{eqnarray}
\end{enumerate}
\end{thm}
The proof of  Theorem \ref{p3} is similar to that in \cite{Puig,Puig06,Amor,LYJFG,Leguil1} with some modifications. We present the proof in the Appendix. We should mention that
it is necessary  to shrink the strip to overcome the small divisor condition (\ref{small}) when we solve the homological equation.
In the proof of cases (i)  and  (ii) of Theorem \ref{p3}, we shrink the strip from $2\delta$ to $\delta$ and $\delta$ to $0$ respectively.



Now we can state our main result of perturbation near the spectral gap.

\begin{thm}\label{pnm}
Let $\delta= 5\beta(\alpha)$ and suppose  $|\epsilon|\leq\frac{1}{C(\alpha)||R||_{2\delta}^4}$. Let $\widehat{R}_{\epsilon}(x)=R_{\epsilon}(x)R_{1,\epsilon}(x)R_{2,\epsilon}(x)\in C^{\omega}(\mathbb{R}/\mathbb{Z},{\rm PSL}(2,\mathbb{R}))$,
where $ R_{1,\epsilon}(x)$  and $R_{2,\epsilon}(x)$ are given by Theorem \ref{p3}.
Then  we have
\begin{equation}\label{pn55}
    \widehat{R}_{\epsilon}^{-1}(x+\alpha)A^{E+\epsilon}(x)\widehat{R}_{\epsilon}(x)=e^{\mathfrak{P}+\epsilon \mathfrak{P}_1+\epsilon^2 \mathfrak{P}_2+\epsilon^3 \mathfrak{R}_\epsilon(x)},
\end{equation}

where
\begin{eqnarray*}
\label{pn56}&&\mathfrak{P}=
\left[
\begin{array}{cc}
  0&\mu_m\\
  0&0
\end{array}
\right],\\
\label{pn57}&&\mathfrak{P}_1=
\left[
\begin{array}{cc}
  -\frac{\mu_m}{2}[R_{11}^2]+[R_{11}R_{12}]&-\mu_m[R_{11}R_{12}]+[R_{12}^2]\\
  -[R_{11}^2]&\frac{\mu_m}{2}[R_{11}^2]-[R_{11}R_{12}]
\end{array}
\right],\\
\label{pn58}&&\mathfrak{P}_2\in{\rm sl}(2,\mathbb{R}),\\
\label{pn60}&&||\mathfrak{P}_2||\leq C(\alpha)||R||_{2\delta}^4,\\
\label{pn63}&&||\mathfrak{R}_\epsilon||_0\leq C(\alpha)||R||_{2\delta}^8.
\end{eqnarray*}
Moreover,
\begin{equation}\label{pn64}
\deg{(\widehat{R}_{\epsilon})}=\deg{(R)}.
\end{equation}

\end{thm}

\begin{proof}
(\ref{pn55}) follows from \eqref{pn11} and some simple computations.

It suffices to prove (\ref{pn64}). From (\ref{pn5}) and (\ref{pn12}), we obtain for $|\epsilon|\leq\frac{1}{C(\alpha)||R||_{2\delta}^4}$
\begin{equation*}
||R_{1,\epsilon}-I||_0\leq \frac{1}{4},||R_{2,\epsilon}-I||_0\leq \frac{1}{4},
\end{equation*}
so that both  $R_{1,\epsilon}$ and $R_{2,\epsilon}$ are homotopic  to the identity. This implies  (\ref{pn64}).
\end{proof}
\section{Proof of the main theorems}
In this section, we will complete the proofs  of Theorems \ref{main}, \ref{maint}, \ref{main1} and \ref{main2}.
We assume $|m| \geq m_{\star}$. Then by (\ref{p1e3}), $n$ is large enough.

\begin{thm}\label{pmt}
Suppose   $0<|\lambda|\leq \lambda_0$. We have
\begin{eqnarray}
\label{pme2}&&E_m^+-E_m^-\leq e^{-\frac{\eta}{3}n}.
\end{eqnarray}
\end{thm}

\begin{proof}
 We let $\delta=5\beta(\alpha)$. From (\ref{p1e2}), one has
\begin{equation*}\label{pme3}
||R||_{2\delta}\leq C_{\star}e^{C\beta(\alpha)n}
\end{equation*}
and so that
\begin{equation}\label{pme4}
\frac{1}{C(\alpha)||R||_{2\delta}^4}\geq e^{-C\beta(\alpha) n}.
\end{equation}

  We define
\begin{equation*}\label{pme17}
\epsilon_m=\frac{-2\mu_m[R_{11}^2]}{[R_{11}^2][R_{12}^2]-[R_{11}R_{12}]^2}<0,
\end{equation*}
 since $\mu_m>0$ by Remark \ref{reliu}.

Following from (\ref{p1e1}) and (\ref{Re6}), we have
\begin{eqnarray}
    \nonumber|\epsilon_m|&\leq&C_{\star}e^{-\frac{\eta}{2} n+C\beta(\alpha)n}\\
    \nonumber&\leq&\frac{1}{C(\alpha)||R||_{2\delta}^4},
\end{eqnarray}
where the second inequality holds by  (\ref{pme4}).

Thus we can apply Theorem \ref{pnm} with $ \epsilon=\epsilon_m<0$. Let
\begin{eqnarray}
   \nonumber \mathfrak{D}&=&\mathfrak{P}+\epsilon_m\mathfrak{P}_1+\epsilon_m^2\mathfrak{P}_2\\
   \nonumber &:=&\left[\begin{array}{cc}
    D_1&D_2\\
   D_3&-D_1 \end{array}\right] \in{\rm sl}(2,\mathbb{R}),
    \end{eqnarray}
    where
    \begin{eqnarray}
   \nonumber &&D_1=\epsilon_m \left([R_{11}R_{12}]-\frac{\mu_m}{2}[R_{11}^2]\right)+O(\epsilon_m^2||\mathfrak{P}_2||),\\
   \nonumber &&D_2=\mu_m+\epsilon_m\left([R_{12}^2]-\mu_m[R_{11}R_{12}]\right)+O(\epsilon_m^2||\mathfrak{P}_2||),\\
   \nonumber &&D_3=-\epsilon_m[R_{11}^2]+O(\epsilon_m^2||\mathfrak{P}_2||),
    \end{eqnarray}
 and
    \begin{equation*}
        \Delta=\det{(\mathfrak{D})}=\frac{{\epsilon_m}^2}{2} ([R_{11}^2][R_{12}^2]-[R_{11}R_{12}]^2)+O(|\epsilon_m|^3||R||_0^2||\mathfrak{P}_2||^2+\mu_m\epsilon_m^2||R||_0^4||\mathfrak{P}_2||).
    \end{equation*}

Recalling (\ref{Re7}) and by direct computations, one has
\begin{eqnarray}
\nonumber |D_1|&\leq&e^{C\beta(\alpha)n}\mu_m,\\
\nonumber |D_2|&\geq&e^{-C\beta(\alpha)n}\mu_m, D_2<0,\\
\nonumber \Delta&\geq& e^{-C\beta(\alpha)n} \mu_m^2>0.
\end{eqnarray}


Let
\begin{equation*}
Q
=\left[
\begin{array}{cc}
0&\frac{\sqrt{-D_2}}{\Delta^{\frac{1}{4}}}\\
\frac{-\Delta^{\frac{1}{4}}}{\sqrt{-D_2}}&\frac{{D_1}}{\Delta^{\frac{1}{4}}\sqrt{-D_2}}
\end{array}
\right],\ Q^{-1}=\left[\begin{array}{cc}
\frac{{D_1}}{\Delta^{\frac{1}{4}}\sqrt{-D_2}}&-\frac{\sqrt{-D_2}}{\Delta^{\frac{1}{4}}}\\
\frac{\Delta^{\frac{1}{4}}}{\sqrt{-D_2}}&0\end{array}
\right].
\end{equation*}
Then
\begin{equation*}
Q^{-1}\mathfrak{D}Q=\left[
\begin{array}{cc}
0&-\sqrt{\Delta}\\
\sqrt{\Delta}&0
\end{array}
\right]
.
\end{equation*}
 Using the estimates  that
\begin{equation*}
    ||Q||, ||Q^{-1}||\leq \frac{1}{\Delta^{\frac{1}{4}}\sqrt{-D_2}},
\end{equation*}
 we obtain
\begin{equation}\label{pme18}
(\widehat{R}_{\epsilon_m}(x+\alpha){Q})^{-1}A^{E_m^{+}+\epsilon_m}(x)\widehat{R}_{\epsilon_m}(x){Q}=
e^{\sqrt{\Delta}\left(
\left[
\begin{array}{cc}
0&-1\\
1&0
\end{array}
\right]+\epsilon_m^3 \mathfrak{S}(x)
\right)},
\end{equation}
where
\begin{equation*}
\mathfrak{S}(x)=\frac{Q^{-1}(\mathfrak{R}_{\epsilon_m}(x))Q}{\sqrt{\Delta}}
\end{equation*}
and
\begin{eqnarray}
 \nonumber ||\epsilon_m^3\mathfrak{S}||_0&\leq&  C_{\star}e^{C\beta(\alpha)n}{\frac{|\epsilon_m|^3||R||_{2\delta}^8}{\mu_m^2}} \\
    &\leq& e^{-\frac{\eta}{4} n} \ll 1.\label{pme19}
\end{eqnarray}

Let $\rho'$ be the fibered rotation number of the right hand side of (\ref{pme18}).   Then $\rho'$ is small and $\rho'\neq0$  by (\ref{rr}) and (\ref{pme19}). Recalling (\ref{pr2}), (\ref{pn64}) and (\ref{pme18}), we have $$2\rho_{\lambda,\alpha}(E_m^++\epsilon_m)=2\rho'+\deg{(R)}\alpha \text{ mod } \mathbb{Z}$$ and  $$2\rho_{\lambda f,\alpha}(E_m^+)=\deg{(R)}\alpha \text{  mod } \mathbb{Z}.$$
This means  $\rho_{\lambda f,\alpha}(E_m^++\epsilon_m)\neq\rho_{\lambda f,\alpha}(E_m^+)$. Then $E_m^++\epsilon_m\notin G_{m}$, that is $$E_m^+-E_m^-\leq|\epsilon_m|\leq e^{-\frac{\eta}{3}n}.$$
\end{proof}
\
\begin{proof}[\bf Proof of Theorem \ref{main}]
In Lemma \ref{Blol2}, let $h=C_2\eta$.
Theorem \ref{main} follows from Theorem \ref{pmt} and the fact that $|m|\leq Cn$ by \eqref{p1e3}.
\end{proof}

\begin{proof}[\bf Proof of Theorem \ref{maint}]
 Theorem \ref{maint} follows from Theorem \ref{main} and the fact that any trigonometric polynomial  is analytic on $\mathbb{C}$.
\end{proof}

\begin{proof}[\bf Proof of Theorem \ref{main1}] For AMO, by Lemma \ref{Blol2},  $\lambda_0= e^{-C_2\eta}$ with $\eta>C_1\beta(\alpha)$.
Let $\eta=\frac{-\ln|\lambda|}{C_2}$ so that $|\lambda|\leq \lambda_0$.
By Theorem \ref{pmt},
we have for $|m|\geq m_{\star}$,
\begin{eqnarray}
\nonumber E_m^+-E_m^-&\leq&  e^{-\frac{1}{3}\eta n}\\
\nonumber&\leq&  |\lambda|^{\frac{1}{3C_2}  n}\\
 &\leq&  |\lambda|^{\frac{1}{3CC_2}  |m|},\label{up}
\end{eqnarray}
where the third inequality holds by (\ref{p1e3}). This implies Theorem \ref{main1}.
\end{proof}
In order to prove Theorem \ref{main2}, we need two lemmas.

\begin{lem}[Corollary 6.1, \cite{LYJMP}]\label{hl1}
Let $ |\lambda|\leq \lambda_0$. Then
\begin{equation}\label{hle1}
|\rho_{\lambda f,\alpha}(E_1)-\rho_{\lambda f,\alpha}(E_2)|\leq C_{\star}|E_1-E_2|^{\frac{1}{2}}, \text{ for all }\  \  E_1,E_2\in\mathbb{R}.
\end{equation}
\end{lem}

\begin{lem}\label{hl2}
Let     $G_m=(E_m^-,E_m^+)$ for  $m\in\mathbb{Z}\setminus\{0\}$ and $G_0=(-\infty,E_{\min})\cup (E_{\max},+\infty)$. Then for $m'\neq m\in\mathbb{Z}\setminus\{0\}$ with $|m'|\geq|m|$, we have
\begin{equation}\label{hle2}
\mathrm{dist}(G_m,G_{m'})=\inf\limits_{x\in G_m,x'\in G_{m'}}{|x-x'|}\geq c_{\star} e^{-8\beta(\alpha)|m'|},
\end{equation}
and for $m\in\mathbb{Z}\setminus\{0\}$
\begin{equation}\label{hle3}
\mathrm{dist}(G_m,G_0)\geq c_{\star} e^{-8\beta(\alpha)|m|}.
\end{equation}
\end{lem}

\begin{proof}
We start with the proof of (\ref{hle2}). From the small divisor condition (\ref{small}), one has
\begin{eqnarray}
\nonumber||(m-m')\alpha||_{\mathbb{R}/\mathbb{Z}}&\geq& \frac{1}{C(\alpha)}e^{-2\beta(\alpha)|m-m'|}\\
\label{hle4}&\geq& \frac{1}{C(\alpha)}e^{-4\beta(\alpha)|m'|},
\end{eqnarray}
for $|m'|\geq |m|$.

Without loss of generality, we assume $E_m^+\leq E_{m'}^-$. By Lemma \ref{hl1}, we have
\begin{eqnarray}
\nonumber\mathrm{dist}(G_m,G_{m'})&=&|E_{m'}^--E_m^+|\\
\nonumber&\geq& \left(\frac{1}{C_{\star}}|\rho_{\lambda f,\alpha}(E_{m'}^-)-\rho_{\lambda f,\alpha}(E_m^+)|\right)^2\\
\nonumber&\geq& c_{\star} ||(m-m')\alpha||_{\mathbb{R}/\mathbb{Z}}^2,\\
&\geq& c_{\star}e^{-8\beta(\alpha)|m'|},
\end{eqnarray}
where the second inequality holds by \eqref{liugap} and the third inequality holds by (\ref{hle4}).
We finish the proof of  (\ref{hle2}).  The proof of (\ref{hle3}) is similar.
\end{proof}

\begin{proof}[\bf Proof of Theorem \ref{main2}]
Let $\eta=C_1\beta(\alpha)$.
We assume $0<\sigma\leq \sigma_{\star}(\lambda,f,\alpha,\epsilon)$.
For $E\in\Sigma_{\lambda f,\alpha}$ and $ \sigma $, let
$$\mathcal{R}(E,\sigma)=\{m\in\mathbb{Z}\setminus\{0\}: (E-\sigma,E+\sigma)\cap G_m\neq \emptyset\}.$$
Define $m_0\in\mathbb{Z}\setminus\{0\}$ with $|m_0|=\min\limits_{m\in\mathcal{R}(E,\sigma)}|m|$. For any $m\in\mathcal{R}(E,\sigma)$, one has
\begin{equation}\label{hle6}
\mathrm{dist}(G_m,G_{m_0})\leq 2\sigma.
\end{equation}

We first assume $(E-\sigma,E+\sigma)\cap G_0=\emptyset$. Recalling (\ref{hle2}), we have for any $m\in\mathcal{R}(E,\sigma)$ with $m\neq m_0$,
\begin{equation*}
2\sigma\geq c_{\star}e^{-8\beta(\alpha)|m|},
\end{equation*}
that is 
\begin{equation}\label{hle7}
|m|\geq \frac{-\ln{(C_{\star}\sigma)}}{8\beta(\alpha)}.
\end{equation}
Then by (\ref{pme2}), we obtain
\begin{equation*}
   \sum\limits_{m\in\mathcal{R}(E,\sigma),m\neq m_0}\mathrm{Leb}((E-\sigma,E+\sigma)\cap G_m)\;\;\;\;\;\;\;\;\;\;\;\;\;\;\;\;\;\;\;\;\;\;\;\;\;\;\;\;\;\;\;\;
\end{equation*}
\begin{eqnarray}
\nonumber\;\;\;\;\;\;\;\;\;\;\;\;\;\;\;\;&\leq&\sum\limits_{m\in\mathcal{R}(E,\sigma),m\neq m_0}(E_m^+-E_m^-)\\
\nonumber\;\;\;\;\;\;\;\;\;\;\;\;\;\;\;\;&\leq& \sum\limits_{|m|\geq\frac{-\ln{(C_{\star}\sigma)}}{8\beta(\alpha)}}C_{\star}e^{-c\eta|m|}\\
\label{hle8}\;\;\;\;\;\;\;\;\;\;\;\;\;\;\;\;&\leq&\epsilon\sigma.
\end{eqnarray}
On the other hand, $E\in\Sigma_{\lambda f,\alpha}$ implies $E\notin G_{m_0}$. Thus  we have
\begin{equation}\label{hle9}
\mathrm{Leb}((E-\sigma,E+\sigma)\cap G_{m_0})\leq\sigma.
\end{equation}
In this case,   (\ref{hle8}) and (\ref{hle9}) imply
\begin{eqnarray}
\nonumber&&\mathrm{Leb}((E-\sigma,E+\sigma)\cap \Sigma_{\lambda,\alpha})\\
\nonumber&\geq&2\sigma-\mathrm{Leb}((E-\sigma,E+\sigma)\cap G_{m_0})\\
&&-\sum\limits_{m\in\mathcal{R}(E,\sigma),m\neq m_0}\mathrm{Leb}((E-\sigma,E+\sigma)\cap G_m)\\
\label{hle10}\nonumber\;\;\;\;\;\;\;\;\;&\geq&2\sigma-\sigma-\epsilon\sigma\geq(1-\epsilon)\sigma.
\end{eqnarray}

In the case $(E-\sigma,E+\sigma)\cap G_0\neq\emptyset$, without loss of generality, we assume $(E-\sigma,E+\sigma)\cap (-\infty,E_{\min})\neq\emptyset$. Then  we have  $$0< E_m^--E_{\min}\leq2\sigma$$
for any $m\in\mathcal{R}(E,\sigma)$. Thus, (\ref{hle7}) also holds for any $  m\in \mathcal{R}(E,\sigma)$ by (\ref{hle3}). From the proof of (\ref{hle8}), we have
\begin{equation}\label{hle11}
\sum\limits_{m\in\mathcal{R}(E,\sigma)}\mathrm{Leb}((E-\sigma,E+\sigma)\cap G_m)\leq \epsilon\sigma.
\end{equation}
Noticing that  $E\in\Sigma_{\lambda f,\alpha}$ and  $E\notin G_0$, one has
\begin{equation}\label{liu7}
   \mathrm{Leb}((E-\sigma,E+\sigma)\cap G_0)\leq\sigma.
\end{equation}
By
  (\ref{hle11})  and \eqref{liu7}, we obtain
\begin{eqnarray}
&&\nonumber\mathrm{Leb}((E-\sigma,E+\sigma)\cap \Sigma_{\lambda,\alpha})\\
\nonumber \;\;\;\;\;\;\;\;\;\;\;\;\;\;\;\;\;\;\;\;&\geq&2\sigma-\mathrm{Leb}((E-\sigma,E+\sigma)\cap G_0)\\
\nonumber&&-\sum\limits_{m\in\mathcal{R}(E,\sigma)}\mathrm{Leb}((E-\sigma,E+\sigma)\cap G_m)\\
\nonumber \;\;\;\;\;\;\;\;\;\;\;\;\;\;\;\;\;\;\;\;&\geq&2\sigma-\sigma-\epsilon\sigma\geq(1-\epsilon)\sigma.
\end{eqnarray}
Putting all the cases together, we complete the proof of    Theorem \ref{main2}.
\end{proof}

\section*{appendix}

\begin{proof}[\bf Proof of Theorem \ref{p3}]
By (\ref{pn3}), we have the upper bound
\begin{equation}\label{liu5}
     ||\widetilde{P}||_{2\delta}\leq C||R||_{2\delta}^2.
\end{equation}

(i). Notice that \eqref{pn6} follows from \eqref{pn10} and \eqref{liu5} directly.
It  suffices to prove \eqref{pn5} and \eqref{pn7}.  The strategy employs Newton's iteration. Let us consider the  cocycles of the form
\begin{equation}\label{pn17}
R_{1,\epsilon}(x)=e^{\epsilon \mathfrak{Y}(x)},
\end{equation}
where $\mathfrak{Y}(x)\in  {\rm sl}(2,\mathbb{R})$  will be specified later. Under the conjugacy of $R_{1,\epsilon}(x)$ in (\ref{pn17}), we  have
\begin{eqnarray}
\nonumber&&e^{-\epsilon \mathfrak{Y}(x+\alpha)}(P+\epsilon\widetilde{P}(x))e^{\epsilon \mathfrak{Y}(x)}\\
\label{pn18}&&=\left(I-\epsilon\mathfrak{Y}(x+\alpha)+O(\epsilon^2)\right)\left(P+\epsilon \widetilde{P}(x)\right)
\left(I+\epsilon\mathfrak{Y}(x)+O(\epsilon^2)\right).
\end{eqnarray}

In order to make the  nonconstant terms of order $\epsilon$ in (\ref{pn18}) vanish, we need to solve
 \begin{equation}\label{pn20}
\mathfrak{Y}(x+\alpha)P-P\mathfrak{Y}(x)=\widetilde{P}-[\widetilde{P}].
\end{equation}

We can solve equation (\ref{pn20}) by using Fourier coefficients.
For this reason,
 let
\begin{eqnarray}
\label{pn23}&&\widehat{\mathfrak{Y}}_{21}(k)=\frac{\widehat{\widetilde{P}_{21}}(k)}{e^{2k\pi i\alpha}-1}\ (k\neq 0), \\
\label{pn24}&&\widehat{\mathfrak{Y}}_{11}(k)=\frac{\mu_m\widehat{\widetilde{P}_{21}}(k)+(e^{2k\pi i\alpha}-1)\widehat{\widetilde{P}_{11}}(k)}{(e^{2k\pi i\alpha}-1)^2}\ (k\neq 0),\\
\label{pn25}&&\widehat{\mathfrak{Y}}_{22}(k)=\frac{(e^{2k\pi i\alpha}-1)\widehat{\widetilde{P}_{22}}(k)-\mu_m e^{2k\pi i\alpha}\widehat{\widetilde{P}_{21}}(k)}{(e^{2k\pi i\alpha}-1)^2}\ (k\neq 0),\\
\label{pn26}&&\widehat{\mathfrak{Y}}_{12}(k)=\frac{\widehat{\widetilde{P}_{12}}(k)+\mu_m\left(\widehat{\mathfrak{Y}}_{22}(k)-e^{2k\pi i\alpha}\widehat{\mathfrak{Y}}_{11}(k)\right)}{e^{2k\pi i\alpha}-1}\ (k\neq 0),
\end{eqnarray}
and
\begin{equation}\label{pn27}
\widehat{\mathfrak{Y}}_{ij}(0)=0\ \mbox{(for any $1\leq i,j\leq 2$)},
\end{equation}
 where $\mathfrak{Y}(x)=(\mathfrak{Y}_{ij}(x))_{1\leq i,j\leq2}$, $\widetilde{P}(x)=(\widetilde{P}_{ij}(x))_{1\leq i,j\leq2}$
 and
\begin{eqnarray*}
&&\mathfrak{Y}_{ij}(x)=\sum_{k\in\mathbb{Z}}\widehat{\mathfrak{Y}}_{ij}(k)e^{2\pi k i x},
\widetilde{P}_{ij}(x)=\sum_{k\in\mathbb{Z}}\widehat{\widetilde{P}_{ij}}(k)e^{2\pi k i x}.
\end{eqnarray*}
It is  easy to  check  that $\mathfrak{Y}(x)$  given by (\ref{pn23})-(\ref{pn26}) solves  (\ref{pn20}) and  belongs to ${\rm sl}(2,\mathbb{R})$.

From the small divisor condition (\ref{small}) and (\ref{pn23})-(\ref{pn26}), we have
\begin{eqnarray}\label{pn42}
||\mathfrak{Y}||_{\delta} \leq C(\alpha)||\widetilde{P}||_{2\delta}\leq C(\alpha)||R||_{2\delta}^2,
\end{eqnarray}
where the second equality holds by \eqref{liu5}.

By the definition of $\mathfrak{Y}(x)$ and (\ref{pn17}), one has
\begin{equation*}
R_{1,\epsilon}^{-1}(x+\alpha)(P+\epsilon\widetilde{P}(x))R_{1,\epsilon}(x) =P_{1,\epsilon}+\epsilon^2\widetilde{P}_{1,\epsilon}(x).
\end{equation*}
Now we are in the position to get the estimates.

Assume
\begin{equation}\label{pn45}
|\epsilon|\leq\frac{1}{C(\alpha)||R||_{2\delta}^2}.
\end{equation}
From (\ref{pn42}) and (\ref{pn45}), one has
\begin{equation*}
|\epsilon|\cdot||\mathfrak{Y}||_\delta\leq c(\alpha),
\end{equation*}
and then
\begin{eqnarray}
\nonumber||R_{1,\epsilon}-I||_\delta
\nonumber&\leq&\sum_{k=1}^{+\infty}\frac{\epsilon^k||\mathfrak{Y}||_{\delta}^{k}}{k!}\\
\nonumber&\leq&C|\epsilon|\cdot||\mathfrak{Y}||_\delta\\
\label{pn47}&\leq& C(\alpha)||R||_{2\delta}^2|\epsilon|,
\end{eqnarray}
which implies (\ref{pn5}). 
 By direct computations, we obtain
\begin{eqnarray}
\nonumber&&\epsilon^2\widetilde{P}_{1,\epsilon}(x)\\
\label{pn48}&&=\epsilon^2(\widetilde{P}(x)\mathfrak{Y}(x)-\mathfrak{Y}(x+\alpha)P\mathfrak{Y}(x)
-\mathfrak{Y}(x+\alpha)\widetilde{P}(x)-\epsilon\mathfrak{Y}(x+\alpha)\widetilde{P}(x)\mathfrak{Y}(x))\\
\label{pn49}&&+\sum_{k=2}^{+\infty}\frac{(-\epsilon)^k\mathfrak{Y}^{k}(x+\alpha)}{k!}
\left(P+\epsilon\widetilde{P}(x)\right)e^{\epsilon\mathfrak{Y}(x)}\\
\label{pn50}&&+\left(1-\epsilon\mathfrak{Y}(x+\alpha)\right)\left(P+\epsilon\widetilde{P}(x)\right)
\sum_{k=2}^{+\infty}\frac{\epsilon^k\mathfrak{Y}^{k}(x)}{k!}.
\end{eqnarray}
 By (\ref{liu5}) and  (\ref{pn42}), we obtain the following estimates,
\begin{eqnarray*}
||(\ref{pn48})||_\delta&\leq&\epsilon^2(2||\widetilde{P}||_\delta\cdot||\mathfrak{Y}||_\delta+||P||\cdot||\mathfrak{Y}||_\delta^2
+|\epsilon|\cdot||\widetilde{P}||_\delta\cdot||\mathfrak{Y}||_\delta^2)\\
&\leq&C(\alpha)||R||_{2\delta}^4\epsilon^2,\\
\end{eqnarray*}
 \begin{eqnarray*}
||(\ref{pn49})||_\delta&\leq&(e^{|\epsilon|\cdot||\mathfrak{Y}||_\delta}-1-|\epsilon|\cdot||\mathfrak{Y}||_\delta)
(2+|\epsilon|\cdot||\widetilde{P}||_\delta)e^{|\epsilon|\cdot||\mathfrak{Y}||_\delta}\\
&\leq&C\epsilon^2\cdot||\mathfrak{Y}||_\delta^2(2+|\epsilon|\cdot||R||_\delta^2)
(1+|\epsilon|\cdot||\mathfrak{Y}||_\delta)\\
&\leq&C(\alpha)||R||_{2\delta}^4\epsilon^2,
\end{eqnarray*}
 and
 \begin{eqnarray*}
||(\ref{pn50})||_\delta &\leq& C(\alpha)||R||_{2\delta}^4\epsilon^2.
\end{eqnarray*}
This  implies (\ref{pn7}).

(ii). The proof of (ii) is similar to the proof  (i). Let
\begin{equation*}
R_{2,\epsilon}(x)=e^{\epsilon^2\mathfrak{X}(x)}, \mathfrak{X}(x)\in{\rm sl}(2,\mathbb{R}),
\end{equation*}
and the homological equation becomes
\begin{equation*}
\mathfrak{X}(x+\alpha)P-P\mathfrak{X}(x)=P_{1,\epsilon}^{\ast}(x),
\end{equation*}
where
\begin{equation*}
P_{1,\epsilon}^{\ast}(x)=\widetilde{P}_{1,\epsilon}(x)-[\widetilde{P}_{1,\epsilon}].
\end{equation*}
Thus under the conjugacy $R_{2,\epsilon}$, we have
\begin{equation*}
R_{2,\epsilon}^{-1}(x+\alpha)(P_{1,\epsilon}+\epsilon^2\widetilde{P}_{1,\epsilon}(x))R_{2,\epsilon}(x) =P_{2,\epsilon}+\epsilon^3 \widetilde{P}_{2,\epsilon}(x),
\end{equation*}
and the  estimates are similar to those in (i).

\end{proof}
\subsection*{Acknowledgment}
The authors   would like to thank Svetlana
Jitomirskaya      for comments on earlier versions of the manuscript. The authors wish to express their gratitude to the  anonymous referee, whose
extraordinary care and thoroughness in reading the manuscript greatly helped the
exposition of the main ideas.
W.L. was supported by the AMS-Simons Travel Grant 2016-2018, NSF DMS-1401204 and NSF DMS-1700314.


\begin{thebibliography}{10}

\bibitem{AA2008}
A.~Avila.
\newblock The absolutely continuous spectrum of the almost {M}athieu operator.
\newblock {\em arXiv preprint arXiv:0810.2965}, 2008.

\bibitem{avila2010almost}
A.~Avila.
\newblock Almost reducibility and absolute continuity {I}.
\newblock {\em arXiv preprint arXiv:1006.0704}, 2010.

\bibitem{Global}
A.~Avila.
\newblock Global theory of one-frequency {S}chr\"odinger operators.
\newblock {\em Acta Math.}, 215(1):1--54, 2015.

\bibitem{AFK}
A.~Avila, B.~Fayad, and R.~Krikorian.
\newblock A {KAM} scheme for {${\rm SL}(2,\Bbb R)$} cocycles with {L}iouvillean
  frequencies.
\newblock {\em Geom. Funct. Anal.}, 21(5):1001--1019, 2011.

\bibitem{Ten}
A.~Avila and S.~Jitomirskaya.
\newblock The {T}en {M}artini {P}roblem.
\newblock {\em Ann. of Math. (2)}, 170(1):303--342, 2009.

\bibitem{AJa1}
A.~Avila and S.~Jitomirskaya.
\newblock Almost localization and almost reducibility.
\newblock {\em J. Eur. Math. Soc. (JEMS)}, 12(1):93--131, 2010.

\bibitem{AJCMP}
A.~Avila and S.~Jitomirskaya.
\newblock H\"older continuity of absolutely continuous spectral measures for
  one-frequency {S}chr\"odinger operators.
\newblock {\em Comm. Math. Phys.}, 301(2):563--581, 2011.

\bibitem{AJM}
A.~Avila, S.~Jitomirskaya, and C.~A. Marx.
\newblock Spectral theory of extended {H}arper's model and a question by
  {E}rd{\H{o}}s and {S}zekeres.
\newblock {\em Invent. Math.}, 210(1):283--339, 2017.

\bibitem{avila2016second}
A.~Avila, S.~Jitomirskaya, and Q.~Zhou.
\newblock Second phase transition line.
\newblock {\em Mathematische Annalen}, pages 1--15, 2016.

\bibitem{AK06}
A.~Avila and R.~Krikorian.
\newblock Reducibility or nonuniform hyperbolicity for quasiperiodic
  {S}chr\"odinger cocycles.
\newblock {\em Ann. of Math. (2)}, 164(3):911--940, 2006.

\bibitem{ayzduke}
A.~Avila, J.~You, and Q.~Zhou.
\newblock Sharp phase transitions for the almost {M}athieu operator.
\newblock {\em Duke Math. J.}, 166(14):2697--2718, 2017.

\bibitem{CEY}
M.~D. Choi, G.~A. Elliott, and N.~Yui.
\newblock Gauss polynomials and the rotation algebra.
\newblock {\em Invent. Math.}, 99:225--246, 1990.

\bibitem{dam}
D.~Damanik and M.~Goldstein.
\newblock On the inverse spectral problem for the quasi-periodic
  {S}chr{\"o}dinger equation.
\newblock {\em Publ. Math. Inst. Hautes \'Etudes Sci.}, 119(1):217--401, 2014.

\bibitem{dam1}
D.~Damanik, M.~Goldstein, and M.~Lukic.
\newblock The spectrum of a {S}ch{\"o}dinger operator with small quasi-periodic
  potential is homogeneous.
\newblock {\em J. Spectr. Theory}, 6(2):415--427, 2016.

\bibitem{Dina}
E.~I. Dinaburg and J.~G. Sina\u\i.
\newblock The one-dimensional {S}chr\"odinger equation with quasiperiodic
  potential.
\newblock {\em Funkcional. Anal. i Prilo\v zen.}, 9(4):8--21, 1975.

\bibitem{eli}
L.~H. Eliasson.
\newblock Floquet solutions for the {$1$}-dimensional quasi-periodic
  {S}chr\"odinger equation.
\newblock {\em Comm. Math. Phys.}, 146(3):447--482, 1992.

\bibitem{fill}
J.~Fillman and M.~Lukic.
\newblock Spectral homogeneity of limit-periodic {S}chr\"odinger operators.
\newblock {\em J. Spectr. Theory}, 7(2):387--406, 2017.

\bibitem{gold2015}
M.~Goldstein, D.~Damanik, W.~Schlag, and M.~Voda.
\newblock Homogeneity of the spectrum for quasi-perioidic {S}chr{\" o}dinger
  operators.
\newblock {\em J. Eur. Math. Soc. to appear}.

\bibitem{gold2016}
M.~Goldstein, W.~Schlag, and M.~Voda.
\newblock On localization and the spectrum of multi-frequency quasi-periodic
  operators.
\newblock {\em arXiv preprint arXiv:1610.00380}, 2016.

\bibitem{Amor}
S.~Hadj~Amor.
\newblock H\"older continuity of the rotation number for quasi-periodic
  co-cycles in {${\rm SL}(2,\Bbb R)$}.
\newblock {\em Comm. Math. Phys.}, 287(2):565--588, 2009.

\bibitem{ruij}
R.~Han and S.~Jitomirskaya.
\newblock Full measure reducibility and localization for quasiperiodic {J}acobi
  operators: {A} topological criterion.
\newblock {\em Adv. Math.}, 319:224--250, 2017.

\bibitem{Houyou}
X.~Hou and J.~You.
\newblock Almost reducibility and non-perturbative reducibility of
  quasi-periodic linear systems.
\newblock {\em Invent. Math.}, 190(1):209--260, 2012.

\bibitem{Ji}
S.~Jitomirskaya and I.~Kachkovskiy.
\newblock {$L^2$}-reducibility and localization for quasiperiodic operators.
\newblock {\em Math. Res. Lett.}, 23(2):431--444, 2016.

\bibitem{jl18}
S.~Jitomirskaya and W.~Liu.
\newblock Universal hierarchical structure of quasiperiodic eigenfunctions.
\newblock {\em Ann. of Math. (2)}, 187(3):721--776, 2018.

\bibitem{jl18r}
S.~Jitomirskaya and W.~Liu.
\newblock Universal reflective-hierarchical structure of quasiperiodic
  eigenfunctions and sharp spectral transition in phase.
\newblock {\em arXiv preprint arXiv:1802.00781}, 2018.

\bibitem{Ji99}
S.~Y. Jitomirskaya.
\newblock Metal-insulator transition for the almost {M}athieu operator.
\newblock {\em Ann. of Math. (2)}, 150(3):1159--1175, 1999.

\bibitem{JM1982}
R.~Johnson and J.~Moser.
\newblock The rotation number for almost periodic potentials.
\newblock {\em Comm. Math. Phys.}, 84:403--438, 1982.

\bibitem{krasovsky2016central}
I.~Krasovsky.
\newblock Central spectral gaps of the almost {M}athieu operator.
\newblock {\em Comm. Math. Phys.}, 351(1):419--439, 2017.

\bibitem{Leguil1}
M.~Leguil, J.~You, Z.~Zhao, and Q.~Zhou.
\newblock Asymptotics of spectral gaps of quasi-periodic {S}chr\"odinger
  operators.
\newblock {\em arXiv:1712.04700}, 2017.

\bibitem{LYJMP}
W.~Liu and X.~Yuan.
\newblock H\"older continuity of the spectral measures for one-dimensional
  {S}chr\"odinger operator in exponential regime.
\newblock {\em J. Math. Phys.}, 56(1):012701, 21, 2015.

\bibitem{LYJFG}
W.~Liu and X.~Yuan.
\newblock Spectral gaps of almost {M}athieu operators in the exponential
  regime.
\newblock {\em J. Fractal Geom.}, 2(1):1--51, 2015.

\bibitem{MP}
J.~Moser and J.~P\"oschel.
\newblock An extension of a result by {D}inaburg and {S}ina\u\i \ on
  quasiperiodic potentials.
\newblock {\em Comment. Math. Helv.}, 59(1):39--85, 1984.

\bibitem{Puig}
J.~Puig.
\newblock Cantor spectrum for the almost {M}athieu operator.
\newblock {\em Comm. Math. Phys.}, 244(2):297--309, 2004.

\bibitem{Puig06}
J.~Puig.
\newblock A nonperturbative {E}liasson's reducibility theorem.
\newblock {\em Nonlinearity}, 19:355--376, 2006.

\end{thebibliography}
\end{document}